\documentclass[a4paper,12pt]{amsart}

\usepackage{amsmath}  
\usepackage{amsfonts}
\usepackage{times}
\usepackage{helvet} 

\usepackage{array}
\usepackage{graphicx}
\usepackage{color}

\usepackage[swedish,english]{babel}

\usepackage[latin1]{inputenc}

\usepackage{tikz}
\usetikzlibrary{arrows,decorations.pathmorphing,matrix,chains,positioning}

\newcommand{\bilddir}{.}
\newcommand{\R}{\mathbb{R}}

\newcommand{\Z}{\mathbb{Z}}

\newcommand{\bigoh}{\mathcal{O}}
\newcommand{\sinc}{\mbox{sinc}}

\newcommand{\tibeta}{\tilde{\beta}}

\newcommand{\dx}{\frac{dx}{2\pi}}
\newcommand{\dxp}{\frac{dx'}{2\pi}}
\newcommand{\dy}{\frac{dy}{2\pi}}
\newcommand{\dz}{\frac{dz}{2\pi}}

\newtheorem{theorem}{Theorem}
\newtheorem{hypothesis}[theorem]{Hypothesis}
\newtheorem{proposition}[theorem]{Proposition}
\newtheorem{lemma}[theorem]{Lemma}

\newtheorem{remark}[theorem]{Remark}

\title{A Boltzmann model for rod alignment and schooling fish}

\author{Eric Carlen$^{(1)}$, Maria C. Carvalho$^{(2)}$, Pierre Degond$^{(3)}$ and Bernt Wennberg$^{(4,5)}$}

\date{}

\begin{document}

\maketitle

\begin{center}
(1) Department of Mathematics, 
Rutgers University \\
110 Frelinghuysen Rd., Piscataway NJ 08854-8019, USA\\
email: carlen@math.rutgers.edu \\

\vspace{0.2cm} 

(2) Department of Mathematics and CMAF, University of Lisbon,\\
Av. Prof. Gama Pinto 2, 1649-003 Lisbon, Portugal \\
email: mcarvalh@cii.fc.ul.pt \\

\vspace{0.2cm} 

(3) Department of Mathematics, Imperial College London, \\
London SW7 2AZ, United Kingdom \\
email:  pdegond@imperial.ac.uk\\

\vspace{0.2cm} 

(4) Department of Mathematical Sciences, \\
Chalmers University of Technology, SE41296 G\"oteborg, Sweden \\

\vspace{0.2cm} 

(5) Department of Mathematical Sciences, \\
University of Gothenburg, SE41296 G\"oteborg, Sweden\\
email: wennberg@chalmers.se

\end{center}

\begin{abstract}
We consider a Boltzmann model introduced by Bertin, Droz and Gr\'egoire as a binary interaction model of the Vicsek alignment interaction. This model considers particles lying on the circle. Pairs of particles interact by trying to reach their mid-point (on the circle) up to some noise. We study the equilibria of this Boltzmann model and we rigorously show the existence of a pitchfork bifurcation when a parameter measuring the inverse of the noise intensity crosses a critical threshold. The analysis is carried over rigorously when there are only finitely many non-zero Fourier modes of the noise distribution. In this case, we can show that the critical exponent of the bifurcation is exactly $1/2$.  In the case of an infinite number of non-zero Fourier modes, a similar behavior can be formally obtained thanks to a method relying on integer partitions first proposed by Ben-Na\"im and Krapivsky.
\end{abstract}



\medskip
\noindent
{\bf Acknowledgements} EC acknowledges partial support by U.S. National Science Foundation grant DMS 1201354.  PD is on leave from CNRS, Institut de Math\'ematiques de Toulouse, France, where this research has been partly conducted. PD acknowledges support from the Royal Society and the Wolfson foundation through a Royal Society Wolfson Research Merit Award, from the French 'Agence Nationale pour la Recherche (ANR)'  in the frame of the contract 'MOTIMO' (ANR-11-MONU-009-01) and from NSF kinetic research network Grant DMS11-07444 (KI-net). MCC was partially supported by FCT Project PTDC/MAT/100983/2008. BW was partially supported by the Swedish research council and by the Knut and Alice Wallenberg foundation. EC, MCC \& BW  wish to acknowledge the hospitality of the Institut de Math\'ematiques, Toulouse and EC \& MCC wish to acknowledge the hospitality of the Faculty of Sciences, University of Gothenburg, where this research was partly done.

\section{Introduction}
\label{sec_intro}

This paper is concerned with the study of some interaction mechanisms between large collections of agents subject to social interaction. Specifically, we consider a Boltzmann model introduced in ~\cite{BertinDrozGregoire2006} as 
a binary interaction counterpart of the Vicsek alignment interaction~\cite{Vicsek_etal1995}. The goal of the present work is to study the equilibria of this Boltzmann model and to rigorously show that this model exhibits pitchfork bifurcations (or second order phase transitions). 

Systems of self-propelled particles interacting through local alignment have triggered considerable literature since the seminal work of Vicsek and co-authors \cite{Vicsek_etal1995}. Indeed, this simple model exhibits all the universal features of collective systems observed in nature and in particular, the emergence of symmetry-breaking phase transitions from disorder to globally aligned phases.  We refer for instance to \cite{Aldana_etal_PRL07, Chate_etal_PRE08, Degond_etal_JNonlinearSci13, Degond_etal_preprint13, Frouvelle_Liu_SIMA12, Gretoire_Chate_PRL04} for the study of these phase transitions. A recent review on this ever-growing literature can be found in \cite{Vicsek_Zafeiris_PhysRep12}. The overwhelming majority of references rely on Individual-Based Models (IBM)  or particle models \cite{Baskaran_Marchetti_PRL10, Carrillo_etal_SIMA10, Chate_etal_PRE08, Chuang_etal_PhysicaD07, Cucker_Smale_IEEETransAutCont07, Czirok_etal_PRE96, Ha_Liu_CMS09, Mogilner_etal_JMB03, Motsch_Tadmor_JSP11, Peruani_etal_PRE06}, mostly with applications to animal collective behavior from bacterias to mammals \cite{Aoki_BullJapSocSciFish92, Couzin_etal_JTB02, Gautrais_etal_PlosCB12}. When the number of agents becomes very large, kinetic models \cite{Bellomo_Soler_M3AS12, Bertin_etal_JPhysA09, Bolley_etal_M3AS11, Fornasier_etal_PhysicaD11, Ha_Tadmor_KRM08} or hydrodynamic models \cite{Barbaro_Degond_DCDSB13, Baskaran_Marchetti_PRE08, Bertin_etal_JPhysA09, Degond_Motsch_M3AS08, Degond_etal_MAA13, Degond_etal_Schwartz13, Frouvelle_M3AS12, Ratushnaya_etal_PhysicaA07, Toner_Tu_PRL95, Toner_etal_AnnPhys05} are more efficient and have received an increasing attention in the literature. 

The present work is concerned with a kinetic, Boltzmann-like model which has been proposed as a kinetic version of the Vicsek particle model  in \cite{Bertin_etal_NewJPhys13, BertinDrozGregoire2006, Bertin_etal_JPhysA09}. This model shows strong similarity with a model proposed by Ben-Na\"im and Krapivsky in \cite{BenNaimKrapivsky2006}. A zero-noise version of this model has been studied in \cite{Degond_etal_arXiv:1403.5233} ; it is shown that generically, Dirac deltas are the stable equilibria of this model. Here, we study the noisy version of this model and show that peaked equilibria (i.e. noisy versions of the Dirac deltas) emerge when the noise intensity becomes smaller than a critical value, and that, at the same time, uniform equilibria become unstable. Our rigorous proof is limited to the case where the noise has a finite number of Fourier coefficients, leaving the case of generic noises open. However, some formal results can be found by adapting the method of integer partitions by Ben-Na\"im and Krapivsky \cite{BenNaimKrapivsky2006}.

The main concern of this paper is the following 
Boltzmann equation:
\begin{eqnarray}
\label{eq:fish}
  \partial_t f(t,x_1) 
&=&\int_{-\pi}^{\pi}\int_{-\pi}^{\pi}
f(t,x'_1)f(t,x'_2) 
g(x_1-\hat x'_{12}) \, \beta(|\sin(x'_2-\hat x'_{12})|) \, \frac{dx'_1}{2 \pi} \, \frac{dx'_2}{2 \pi}
 \nonumber \\
&&\qquad\qquad\qquad\qquad - f(t,x_1) \int_{-\pi}^{\pi} f(t,x_2) \,
\beta(|\sin(x_2-\hat x_{12})|) \,\frac{dx_2}{2 \pi}  
. 
\end{eqnarray}
Here, $\hat x_{12} = \mbox{Arg} \{ \frac{e^{i x_1} +e^{ix_2}}{|e^{i x_1} +e^{ix_2}|} \}$ is the argument (modulo $2 \pi$) of the midpoint on the
smallest arc on the unit circle between $e^{i x_1}$ and $e^{i x_2}$, $\hat x'_{12} =
\mbox{Arg} \{ \frac{e^{i x'_1} +e^{ix'_2}}{|e^{i x'_1} +e^{ix'_2}|} \}$. The quantity $2|\sin(x_2-\hat
x_{12})|$ is the euclidean distance in ${\mathbb R}^2$ between
$x_1$ and $x_2$. As usual in kinetic theory, the collision rate
between two particles is a function $\beta$ of this distance. The
unknown $f$ is a probability density on the circle ${\mathbb S}^1
\approx {\mathbb R}/(2 \pi {\mathbb Z})$, giving {\em e.g.} 
the distribution of directions in a fish school, and $g$ is a given
probability density modeling the noise in the model. The first term at
the right-hand side (the gain term) expresses the rate at which
particles acquire the velocity $x_1$ as a result of collisions of two
particles of velocities $x_1'$ and $x_2'$. The post-collision velocity
$x_1$ of particle $1$ is distributed around the ``mid-point''  (in the
sense above) $\hat x'_{12}$ of the two pre-collisional velocities
$x_1'$ and $x_2'$ according to the probability distribution $g$. The
loss term (the second term) is found in a similar way reversing the
roles of the pre- and post-collisional velocities. In our case 
$\beta$ is just a constant (to mimic ``Maxwellian molecules'' in gas
dynamics)  or if one takes a collision rate proportional to the
relative velocities of the particles as usual in kinetic theory,
$\beta(x)$ is proportional to $x$. A space-dependent version of this
equation was first formulated by  E. Bertin, M. Droz and G. Grégoire
in~\cite{BertinDrozGregoire2006} as 
a model for swarm dynamics inspired by the so-called Vicsek model~\cite{Vicsek_etal1995} (see also e.g. \cite{Bertin_etal_NewJPhys13, Bertin_etal_JPhysA09}).  

A rigorous derivation of equation~(\ref{eq:fish}) as a limit as
$N\rightarrow\infty$ of an $N$-particle system was carried out 
in~\cite{CarlenChatelinDegondWennberg2011, CarlenDegondWennberg2011},
where a general {\em propagation of chaos} result is obtained for  {\em pair
  interaction driven $N$-particle systems}. These are defined as
Markov jump 
processes in an $N$-fold product space ${\mathbb T}^N = ({\mathbb
  S}^1)^N$, where jumps almost surely only involve two
coordinates. The jumps are triggered by a Poisson clock with  rate
proportional to $N$, and the outcome of a
jump is independent of the clock. A jump involves first a choice of a
pair $(j,k)$ from the set $1\le j<k \le N$, and then a transition
$x\mapsto x'$, independent of $(j,k)$: 
$$  x=(x_1,....,x_j,...,x_k,....,x_N) \mapsto
  (x_1,....,x_j',...,x_k',....,x_N)=x'\,. $$
The jump process behind equation~(\ref{eq:fish}) is defined in the
$N$-dimensional torus, represented by coordinates $x_j\in [-\pi\pi[$
. The jumps take a pair $(x_j,x_k)$ to  
$$  (x_j',x_k') = (\hat{x}_{jk} +
  X_j,\hat{x}_{jk} + X_k) \qquad \mbox{mod}\qquad
  2\pi\times 2\pi\,, $$
where $X_j$ and $X_k$ are independent and equally distributed angles (see Figure~\ref{fig:001}). 
Of course this is not well defined on the set  $x_j  = -x_k$, but that is a set
of measure zero, and at least if the distribution of $x_j$ has a density, this case may be neglected.
\begin{figure}
  \centering
  \begin{tikzpicture}
    \draw[blue] (0,0) circle (2.0);
    \draw[black,->](-2.4,0)--(2.4,0);
    \draw[black,->](0,-2.4)--(0,2.4);
    \draw[green, fill] (-1.,1.73205)  circle (0.1)  node[anchor=south east] {$x_j$};  
    \draw[green, fill] (1.87939,0.68404)  circle (0.1) node[anchor=north west] {$x_k$}; 

    \draw[red, fill] (0.68404,1.87939) circle (0.1) node[above] {$\hat{x}_{jk}$};

    \draw[red](-0.174311,1.99239)  circle (0.1) node[anchor=south east] {$x_j'$};
    \draw[red] (1.17557,1.61803) circle (0.1) node[anchor=south west] {$x_k'$};
     \end{tikzpicture}
  \caption{The jump process in the BDG model}
  \label{fig:001}
\end{figure}
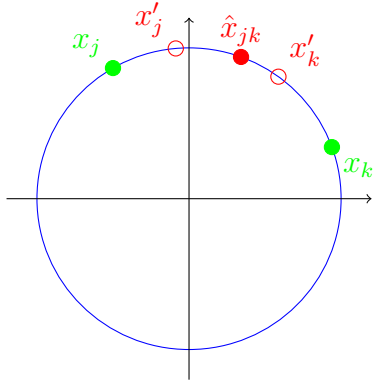

An interesting feature of this process is that, although
propagation of chaos holds, as required for the derivation of
equation~(\ref{eq:fish}), this equation has strongly peaked solutions,
which implies certain dependence between two particles distributed
according to the density $f$. We will expand on this statement below,
where the formal calculations in going from an $N$-particle system to
the kinetic equation are repeated.

The main new results in this paper concern
equation~(\ref{eq:fish}). First, it is easy to see that the uniform 
density, $f(x)=1/{2 \pi}$ is a stationary equilibrium, and that the
(linearized) stability of this equilibrium depends on the first moment
$\gamma_1$ of the noise distribution $g$. The moment $\gamma_1$
indicates how peaked $g$ is (the larger $\gamma_1$, the more strongly
peaked $g$ is).  Second, in the Maxwellian case, we explicitly
construct non-uniform stationary solutions when the noise distribution
$g$ has a finite number of non-zero Fourier coefficients. We prove the
existence of a pitchfork bifurcation (or second-order  phase
transition) when $\gamma_1$ crosses a critical value $\gamma_c =
\pi/4$. For $\gamma_1 \leq \gamma_c$, the uniform stationary
distributions is stable. For $\gamma_1 > \gamma_c$ and close to it,
there exists another class of equilibria which are stable while the
uniform stationary distribution becomes unstable. Additionally, we can
prove that the associated critical exponent is $1/2$ when considering
the first moment of the stationary solution as an order parameter.

An equation very similar to~(\ref{eq:fish}) is studied by Ben-Naim and
Krapivsky in~\cite{BenNaimKrapivsky2006} as a model for rod alignment: 
\begin{eqnarray}
  \label{eq:bennaim}
  \frac{\partial}{\partial t}f(x,v) &=& D \frac{\partial^2}{\partial x^2}f(x,t)
  + \int_{-\pi}^{\pi} f(x+y/2,t) f(x-y/2,t)\,\frac{dy}{2\pi} - f(x,t)\,.  
\end{eqnarray}
While in equation~(\ref{eq:fish}) all particles remain fixed between
the pair interactions, the model of Ben-Naim and Krapivsky assumes
that each particle follows a Brownian motion between the jumps. On the
other hand, contrary to equation~(\ref{eq:fish}), the jumps in
equation~(\ref{eq:bennaim}) imply perfect alignment. More
considerations about this model will be found in
Section~\ref{sec:fourier}, and in particular in
  Section~\ref{sec:BenNaimKrapivsky}, where  the analysis
  in~\cite{BenNaimKrapivsky2006} is studied in more detail.  Their
  analysis also uses the Fourier series
  expansion of the stationary solution, and  semi explicit
  expressions for the Fourier coefficients are obtained by expanding
  these coefficients as a power series of the first coefficient,
  $a_1$. We adapt their method to our case, and at the same time we
  try to clarify some technical points of the method. The result is
  formal in the sense that we do not prove convergence of any of the
  series appearing in the work, but it does provide new insights in
  the behavior of the model.

The layout of the paper is as follows. In Section \ref{sec:real}, we review the simple case where the model is posed on the real line (instead of the circle). In this case, an explicit formula for the equilibria can be found in Fourier-transformed variables. Going back to the model posed on the circle in Section \ref{sec:fourier}, we show that the Fourier coefficients of the distribution function satisfy a fully-coupled nonlinear dynamical system. The linearization of this system about an isotropic equilibrium is studied in Section \ref{sec:linearized}. We show that the isotropic equilibrium is unstable for noise intensities below a certain threshold and that the instability only appears in the first Fourier coefficient, suggesting that the first Fourier mode acts as an order parameter for this symmetry-breaking phase transition. In Section \ref{sec:explicit}, we rigorously prove the emergence of the phase transition and determine the critical exponent in the case where the noise probability has only finitely many non-zero Fourier modes. Indeed, in such a circumstance, any equilibrium solution has also finitely many non-zero Fourier coefficients, and finding such an equilibium can be rigorously accomplished using the Implicit Function Theorem.  We also show that the critical exponent of the phase transition is equal to $1/2$. It is interesting to contrast this result with that of \cite{Degond_etal_preprint13} where all critical exponents between $1/4$ and $1$ were found for the Vicsek dynamics. Removing the assumption of finitely many modes, only formal calculations can be performed at present. The work of Ben-Na\"im and Krapivsky  \cite{BenNaimKrapivsky2006} suggests that the critical exponent $1/2$ persists. In Section \ref{sec:BenNaimKrapivsky}, we relate their integer partition method to our approach. Finally, conclusions and perspectives are drawn in Section \ref{sec:conclu}.

\section{The model on the real line}
\label{sec:real}

\bigskip

In order to get a preliminary sense of the behavior of the model, it is
useful to investigate the more simple case where $x \in {\mathbb
  R}$. In this case, the Boltzmann equation is given by: 
\begin{eqnarray*}
  \partial_t f(t,x_1) 
&=&\int_{-\infty}^{\infty}\int_{-\infty}^{\infty}
f(t,x'_1)f(t,x'_2) 
g(x_1-\hat x'_{12}) \, \beta(|x'_2-\hat x'_{12}|) \, dx'_1 \,  dx'_2
\\
&&\qquad\qquad\qquad\qquad - f(t,x_1) \int_{-\infty}^{\infty} f(t,x_2)
\, \beta(|x_2-\hat x_{12}|) \,dx_2  \, . 
\end{eqnarray*}
where now, $\hat x_{12} = {(x_1+x_2)}/{2}$ and $x_2-\hat x_{12} = {(x_2-x_1)}/{2}$. This corresponds to pair interactions given by
\begin{eqnarray}
  \label{eq:1}
  (x_j,x_k) &\mapsto& \left(\frac{x_j+x_k}{2} +X_1,
    \frac{x_j+x_k}{2} +X_2     \right)\, 
\end{eqnarray}
where $X_1$ and $X_2$ are two independent, identically distributed random variables. The process is then similar to models considered in models of trade~\cite{Cordier_etal_JSP05} and is interesting in the present context mostly because it permits rather explicit calculations. A very similar model was also obtained~\cite{BenNaimKrapivsky2006} as a limit of nearly aligned rods. 

By a simple change of variables $x_2' = x_1' + y$, and using the fact
that we look for $f$ being a probability distribution, the Boltzmann
equation in the Maxwellian case simplifies to:  
\begin{eqnarray*}
  \partial_t f(t,x) 
&=&\int_{-\infty}^{+\infty}\int_{-\infty}^{+\infty}
f(t,x')f(t,x'+y) 
g(x-x'-\frac{y}{2})  \, dx dy - f(t,x) \,.
\end{eqnarray*}
We note that this can be written equivalently as 
$$  \partial_t f = (2 (f\ast f) ( 2 \cdot))\ast g - f\,.$$
Therefore, equilibria are solutions of the fixed-point equation: 
\begin{eqnarray}
\label{eq:line_3}
  f = (2 (f\ast f) ( 2 \cdot))\ast g\,,
\end{eqnarray}
which expresses that the distribution of $\frac{x_1+x_2}{2} + X$ when $x_1$ and $x_2$ are i.i.d. with
density $f$ and $X $ is a random variable of density $g$ must be equal to $f$ itself. 

\begin{theorem}
We suppose that $g \in {\mathcal P}_2 \cap L^1({\mathbb R}) \cap
C^0({\mathbb R}) $ where ${\mathcal P}_2$ is the space of probability
measures of~${\mathbb R}$ with bounded second moments. Additionally,
we suppose that $g$ has zero mean. The solutions in ${\mathcal P}_2
\cap L^1({\mathbb R})$ of (\ref{eq:line_3}) are given by translations by
an arbitrary real number of a probability $f \in  {\mathcal P}_2 \cap
L^1({\mathbb R})$ whose Fourier transform $\hat f(\xi)$ has
the expression: 
$$\hat f(\xi) = \prod_{j=0}^{\infty} \hat{g}(\xi/2^j)^{2^j} .$$
\label{thm:line}
\end{theorem}

\medskip
\noindent
{\bf Proof.}  We define
$$ \hat g_n (\xi) = \prod_{j=0}^{n-1} \, \hat{g}(\xi/2^j)^{2^j}. $$
We note that $\hat g_n$ is the Fourier transform of $g_n$ which
satisfies the recursion for $n \geq 1$: 
\begin{eqnarray}
\label{eq:7_1}
 g_n = g * (2 g_{n-1} (2 \cdot)) * (2 g_{n-1} (2 \cdot)).
\end{eqnarray}
and $g_0 = g$. Now, by recursion, $g_n$ is a probability
density. Indeed, supposing that $g_{n-1}$ is a probability density, we
obtain $g_n$ as the convolution of three probability densities. Now,
we write, uniformly on any compact set for $\xi$: $\hat{g}(\xi) = 1 -
\frac{1}{2} \gamma_2 \xi^2 + o(\xi^2)$, where $\gamma_2 =
\int_{\mathbb R} g(x) \, x^2 \, dx$ is the second moment of $g$. Then,
uniformly for $\xi$ in any bounded interval and $n \in {\mathbb N}$,
we get:  
\begin{eqnarray*}
   \log \hat g_n &=&
    \sum_{j=0}^{n-1} 2^j \log\left( 1 - \frac{1}{2} \gamma_2 (\xi/ 2^j) ^2 +
o((\xi/2^j) ^2)  \right) 
\\
&=& - \frac{1}{2} \gamma_2 \xi^2 \sum_{j=0}^{n-1} 2^{-j}    + O(\xi^2)\,. 
\end{eqnarray*}
Letting $n \to \infty$, we get
$$   \lim_{n \to \infty} \log \hat g_n (\xi) =
     -  \gamma_2 \xi^2     + O(\xi^2)\,, $$
uniformly for $\xi$ in any compact set of ${\mathbb R}$. Hence, this
defines $\hat g_\infty(\xi)$ as a continuous function of $\xi$ which
by Levi's continuity theorem, is the Fourier transform of a
probability measure $g_\infty$. Now, taking $n \to \infty$ in
(\ref{eq:7_1}), we get 
\begin{eqnarray}
\label{eq:7_2}
 g_\infty = g * (2 g_{\infty} (2 \cdot)) * (2 g_{\infty} (2 \cdot)).
\end{eqnarray}
which expresses $g_\infty$ as the convolution of a continuous function
$g$ with a measure $(2 g_{\infty} (2 \cdot)) * (2 g_{\infty} (2
\cdot))$. Therefore, $g_\infty$ is a continuous function and
consequently an element of $L^1({\mathbb R})$. Finally, by a simple
change of variables, (\ref{eq:7_2}) is nothing but
Eq. (\ref{eq:line_3}) with $f=g_\infty$. Therefore, $g_\infty$ is a
solution of  (\ref{eq:line_3}).  

\begin{remark}
The equilibrium distribution $g_\infty$ has a second moment that is
twice that of~$g$. Figure~\ref{fig:3} shows the solution to
equation~(\ref{eq:line_3}) in the case where $g(x) = \frac{1}{2}
1_{[-1,1]}$, where $1_{[-1,1]}$ is the indicator function of the
interval $[-1,1]$. When $g$ is a centered Gaussian, then $f$ is also a
Gaussian with twice its variance.  
\end{remark}

\begin{figure}[h]
  \centering
  \includegraphics[width=0.7\textwidth]{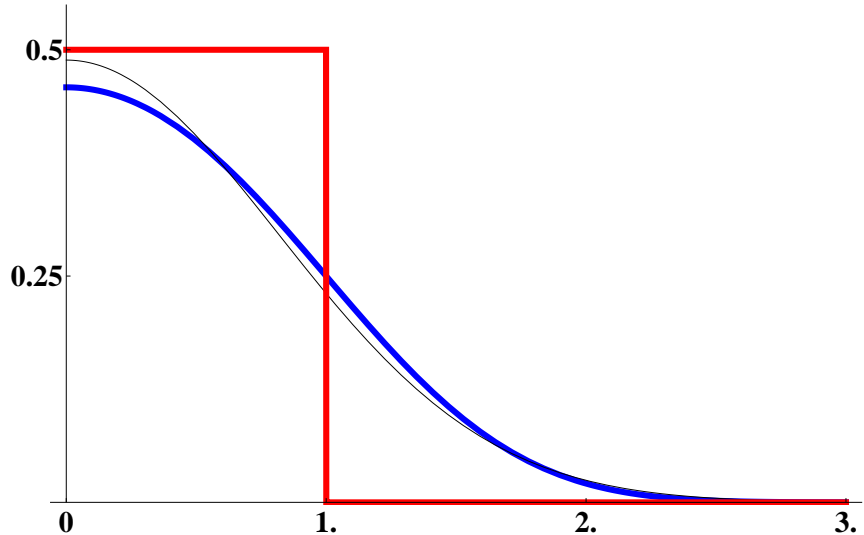}
  \caption{A solution $f$  to equation~(\ref{eq:line_3}) (the blue, thick curve) with $g(x) =
    \frac{1}{2} 1_{[-1,1]}$ (red, thick curve)   compared with
    the Gaussian function with the same variance (the thin curve).}
  \label{fig:3}
\end{figure}
 
\begin{remark}
A model where the pair interacts more weakly can be obtained by
replacing Equation~(\ref{eq:1})  with
\begin{eqnarray*}
(x_j,x_k) &\mapsto& \left(\lambda x_j +(1-\lambda)x_k  +X_1,
    (1-\lambda) x_j+\lambda x_k +X_2     \right) \,.
\end{eqnarray*}
One can then proceed in the same way by taking the Fourier transform to get
\begin{eqnarray*}
\hat{f}(\xi) &=& \hat{f}(\lambda\xi)\hat{f}( (1-\lambda)\xi) \hat{g}(\xi)\,,
\end{eqnarray*}
and as in the case of $\lambda=1/2$ obtain a solution
\begin{eqnarray*}
  \hat{f}(\xi) &=& \prod_{k=0}^{\infty}\prod_{j=0}^{k}
  \hat{g}\left(\lambda^j(1-\lambda)^{k-j} \xi
  \right)^{{k\choose j}}\,.
\end{eqnarray*}

In this case the variance of $f$ can be expressed in terms of the
variance of $g$ as
\begin{eqnarray*}
  \mbox{Var}[f] &=&  \frac{1}{2\lambda(1-\lambda)} \mbox{Var}[g]
\end{eqnarray*}
\end{remark}

Now, we are going to apply the same method to the original model posed on the circle. But we will see that the difficulties are considerably bigger.

\section{Fourier series expansion of the model on the circle} 
\label{sec:fourier}

Now, we are back to model (\ref{eq:fish}) posed on the circle. We first remark that, by the change of variables $x_2' = x_1'+y$, $y \in ]-\pi, \pi]$, we have $\hat x'_{12} = x_1' + y/2$, $x_2'-\hat x'_{12} = y/2$, so that the model can be written:
\begin{eqnarray}
\label{eq:fish_10}
\partial_t f(t,x) 
&=&\int_{-\pi}^{\pi}\int_{-\pi}^{\pi}
\bigg(f(t,x')f(t,x'+y) 
g(x-x'-\frac{y}{2}) 
\nonumber \\
&&\quad\qquad\qquad\qquad - f(t,x)f(t,x+y)\bigg)
\beta(|\sin(y/2)|)\dxp \dy\,.
\end{eqnarray}

Multiplying with a test function $\phi$, integrating
over $[-\pi,\pi]$, and performing a change of variables gives the
following weak form of the equation,
\begin{eqnarray}
\label{eq:fish_w}
\lefteqn{  
\frac{d}{dt} \int_{S^1}f(t,x)\phi(x)\,\dx 
}
&&\nonumber \\
&=&\int_{-\pi}^{\pi}\int_{-\pi}^{\pi}\int_{-\pi}^{\pi}
f(t,x)f(t,x+y)g(z)\tibeta(y)
\left(
\phi(x+y/2 + z) - \phi(x) \right) \frac{dx}{2 \pi} \frac{dy}{2 \pi} \frac{dz}{2 \pi} \, . 
\nonumber\\
\end{eqnarray}
We will only consider the cases where $\tibeta=1$ (Maxwellian molecules) or $\tibeta=|\sin(y/2)|$ (hard-sphere case).

Note that formally the system conserves mass:
$$ \int_{-\pi}^{\pi} f(x,t) \, dx = \mbox{Constant}. $$
We may therefore require that $f(x,t) \, dx$ is a probability, i.e. take this constant equal to unity. 

Because all functions are
periodic, it is natural to consider to rewrite the system in terms of
the Fourier series. Introducing 
\begin{eqnarray*}
  f(x) &=& \sum_{k=-\infty}^{\infty} a_ke^{i k x}\qquad\qquad a_k
  \,=\,\int_{-\pi}^{\pi} f(x) e^{-ikx}\,\dx\,. \\
\gamma_k&=&(2\pi)^{-1}\int_{-\pi}^{\pi} g(z)e^{-ikz}dx, \qquad
\Gamma(u) =(2\pi)^{-1} \int_{-\pi}^{\pi} \tibeta(y)e^{iuy}dy, 
\end{eqnarray*}
we have the following: 

\begin{proposition}
Suppose that $g$ is even and let $a_k(t)$ be the Fourier coefficients of a solution of Eq. (\ref{eq:fish_10}) which is an even probability density. Then,  $a_0 = 1$ and $a_k$ for $k \not = 0$ satisfy $a_{-k}=a_k$ and solve the following system: 
\begin{eqnarray} 
\label{eq:coscoeffeq}
  \frac{d}{dt} a_k(t) &=& \left(\, 2 \gamma_k\Gamma(k/2)-\Gamma(0)-\Gamma(k)\, \right) a_k(t) + \nonumber \\
& &\sum_{n=1}^{k-1}  \left( \gamma_k \Gamma(n-k/2)-\Gamma(n)\right) a_n(t) a_{k-n}(t) +  \nonumber \\
& &\sum_{n=k+1}^{\infty}  \left( 2 \gamma_k \Gamma(n-k/2)-\Gamma(n)-\Gamma(n-k) \right) a_n(t) a_{n-k}(t)
\end{eqnarray}
The function $\Gamma(u)$,  which is to be evaluated only on half-integer
points,  is 
  \begin{eqnarray}
  \label{eq:gammaMax}
    \Gamma(u) &=& \frac{\sin(\pi u)}{\pi u} = 
    \left\{
      \begin{array}{lcl}
        1 &\qquad \mbox{when}\qquad & u=0\\
        0 &\qquad \mbox{when}\qquad & u\in\Z\setminus\{0\}\\
        \frac{2(-1)^{\ell}}{\pi(2 \ell+1)} &\qquad \mbox{when}\qquad u= \ell+1/2 
      \end{array}\right.\nonumber\\
  \end{eqnarray}
in the Maxwellian case, when $\tibeta(1)\equiv 1$; and 
  \begin{eqnarray*}
    \Gamma(u) &=& \frac{2-4 u \sin (\pi  u)}{\pi -4 \pi  u^2}= 
    \left\{
      \begin{array}{lcl}
        2/(\pi(1-4 u^2)) &\qquad \mbox{when}\qquad & u\in \Z\\
        1/\pi  &\qquad \mbox{when}\qquad & u=\pm 1/2\\
        \frac{2 (-1)^{\ell} \ell+(-1)^{\ell}-1}{2 \pi  \ell^2+2 \pi  \ell}&\qquad
        \mbox{when}\qquad& u= \ell+1/2, \ell \ne 0,-1  
      \end{array}\right. \,, 
			\\
&&
  \end{eqnarray*}
 in the hard-sphere case, when $\tibeta(y) = |\sin(y/2)|$. 
\label{prop:Fourier}
\end{proposition}

\medskip
\noindent
{\bf Proof.} Taking $\phi(x)=e^{-ikz}$ in (\ref{eq:fish_w}), we get (with $a_k=a_k(t)$) for $k \not = 0$
\begin{eqnarray}
  \frac{d}{dt} a_k &=&
\sum_{n}\sum_{m} a_m a_n
\int_{-\pi}^{\pi}\int_{-\pi}^{\pi}\int_{-\pi}^{\pi}
e^{i m x} e^{i n (x+y)} g(z)\tibeta(y)
\left(
e^{-ik(x+y/2 + z)} - e^{-ik x} \right) \dx \dy \dz \nonumber \\
&=&\sum_{n}\sum_{m} a_m a_n
\int_{-\pi}^{\pi}\int_{-\pi}^{\pi}\int_{-\pi}^{\pi} g(z)\tibeta(y)
\left(
e^{i ( (m+n-k)x +(n-k/2)y -kz)} - \right. \nonumber \\
&&\qquad\qquad\qquad\qquad \qquad\qquad\qquad\qquad \qquad\qquad\qquad\qquad \left. e^{i( m+n-k)x+ny   )}
\right)\dx \dy \dz \nonumber \\
&=&\sum_{n} a_{k-n} a_n
\int_{-\pi}^{\pi}\int_{-\pi}^{\pi} g(z)\tibeta(y)
\left(
e^{i ((n-k/2)y -kz)} - e^{i( ny   )}
\right) \dy \dz\nonumber 
\end{eqnarray}
which leads to
\begin{eqnarray}
\label{eq:fcoeffeq}
 \frac{d}{dt} a_k(t) &=&  \sum_{n} a_{k-n}(t) a_n(t) \left( \gamma_k
   \Gamma(n-k/2) - \Gamma(n)\right)\,\nonumber \\
&=& \sum_{i+j=k} a_i(t)a_j(t)  \left( \gamma_k
   \Gamma((j-i)/2) - \Gamma(j)\right)\,.
\end{eqnarray}
Using that $\gamma_{-k} = \gamma_k$ and $a_{-k} = a_k$, we get (\ref{eq:coscoeffeq}). 

%

\begin{remark}
Eq.~(\ref{eq:coscoeffeq}) for the Maxwellian case can be simplified and gives:
\begin{eqnarray*}
  \frac{d}{dt} a_k(t) &=& \left(\, 2 \gamma_k\Gamma(k/2)-1\, \right) a_k(t) + 
	\\ 
& &\sum_{n=1}^{k-1}  \gamma_k \Gamma(n-k/2) a_n(t) a_{k-n}(t) + 
\\
& &\sum_{n=k+1}^{\infty}  2 \gamma_k \Gamma(n-k/2) a_n(t) a_{n-k}(t)
\end{eqnarray*}
\end{remark}

\begin{remark}
For comparison, we note that the Fourier coefficients of solutions to
equation~(\ref{eq:bennaim}) satisfy
\begin{eqnarray*}
  \frac{d}{dt}a_k(t) &=& -(1+D k^2) a_k(t) + \sum_{i+j=k} \Gamma( (i-j)/2
  ) a_j(t) a_i(t)\,,
\end{eqnarray*}
with $\Gamma$ as in equation~(\ref{eq:gammaMax}) (see
\cite{BenNaimKrapivsky2006}). The only essential difference with
equation~(\ref{eq:fcoeffeq}) is that  the diffusion term manifests
itself as a multiplier $D k^2$ of $a_k$ (and moreover that
(\ref{eq:fcoeffeq}) includes the possibility of non-Maxwellian interactions). 
\end{remark}

\section{The linearized equation}
\label{sec:linearized}

 It is easy to verify that $f(x)\equiv 1$
is a solution, which corresponds to  $a_0=1, a_k=0$, 
$(k\ne0)$. If $f$ is a solution, then any
translation of $f$, i.e. $x\mapsto f(x+s)$) is also a
solution. Expressed in terms of the Fourier coefficients, this means
that if $(a_k)_{k\in\Z}$ is a solution, then so is $(a_k
e^{iks})_{k\in\Z}$.

To investigate the stability of the uniform density, let $f(x,t) = 1 +
\varepsilon F(x,t)$, and let $b_k(t),\;k\in\Z$ be the Fourier
coefficients of $F(x,t)$. Then $b_0=0$, and for $k\ne 0$,
\begin{eqnarray*}
  \frac{d}{dt} b_k(t) &=& b_k(t) \left( 2 \gamma_k \Gamma(k/2)-\Gamma(0)-\Gamma(k)\right)\,.
\end{eqnarray*}
Hence the linearized stability may be determined by analyzing
separately the sign of $\mbox{Re} \lambda_k$ where
\begin{equation}
\lambda_k = (2\gamma_k \Gamma(k/2)  -\Gamma(0)-\Gamma(k)). 
\label{eq:stab}
\end{equation}
Indeed, if $\mbox{Re} \lambda_k \leq 0$, $\forall k \in {\mathbb Z}$, the system is stable, and it is unstable otherwise. Note that $\lambda_0 = 0$ and $\lambda_k \in {\mathbb R}$, $\forall k \in {\mathbb Z}$ in our case. 

\begin{remark}
The uniform density is also stationary for the model
in~\cite{BenNaimKrapivsky2006}, where its stability is analyzed in
very much the same way, giving an explicit expression involving the
only parameter in the model, the diffusion coefficient $D$.
\end{remark}

We assume that $g$ is even. In both the Maxwellian and hard-sphere case, we have the:

\begin{theorem}
We have $ \lambda_k \leq 0$, $\forall k \in {\mathbb Z}$, $|k|\geq 2$, meaning that the linearized stability depends only on the sign of $\lambda_1 = \lambda_{-1}$:
\begin{eqnarray*}
\mbox{the system is stable} & \Longleftrightarrow & \lambda_1 \leq 0 \\
\end{eqnarray*}
\label{thm:linearized_stability}
\end{theorem}

\noindent
{\bf Proof.}  In the Maxwellian case, we have
\begin{eqnarray*}
  2\Gamma(k/2)-\Gamma(0)-\Gamma(k) &=& 
\frac{4 \sin \left(\frac{k \pi }{2}\right)}{k \pi   }-1 . 
\end{eqnarray*}
It is easily seen that the right-hand side is negative when $|k| \geq 2$. Hence it is only $\lambda_1$ that may become positive, and therefore the condition for stability of the uniform solution is that $\gamma_1 \le \frac{\pi}{4}$.

In the hard-sphere case, we find that $2 \Gamma(1/2) -\Gamma(0)-\Gamma(1)= 2/(3\pi)$, and that for $k>1$,
\begin{eqnarray*}
   2\Gamma(k/2)-\Gamma(0)-\Gamma(k) &=& -\frac{4 \left(2 k^4-4 \sin \left(\frac{k \pi
   }{2}\right) k^3+k^2+\sin \left(\frac{k \pi
   }{2}\right) k\right)}{\left(k^2-1\right)
   \left(4 k^2-1\right) \pi }
\end{eqnarray*}
Because $\Gamma$ is an even function, it is enough to consider $k\ge2$, and in that case the numerator is larger than
\begin{eqnarray*}
4 \left(2 k^4-4 \sin \left(\frac{k \pi
   }{2}\right) k^3+k^2+\sin \left(\frac{k \pi
   }{2}\right) k\right) &\ge&
4 \left(2 k^4-4 k^3+k^2- k\right) \\
&\ge& 4(k^2-k)>0\,
\end{eqnarray*}
and hence we may deduce that $\lambda_k<0$ for $|k|>1$ also in this case. If $\gamma_k$ changes sign the calculation is more complicated, but the result is the same: it is only the first Fourier modes of the solution $f$ that may cause instability of the uniform stationary states.

\medskip
For concreteness, we now consider a family of distributions $g(y)$ defined as the periodization of $\frac{1}{\tau}\rho(\frac{y}{\tau})$, where $\rho$ is a given even probability density on $\R$:
\begin{eqnarray*}
  g_{\tau}(y) &=& 2\pi\sum_{j=-\infty}^{\infty}
  \frac{1}{\tau}\rho(\frac{y-2\pi j}{\tau}) . 
\end{eqnarray*}
Then 
\begin{eqnarray*}
  \gamma_k(\tau) &=& \int_{-\pi}^{\pi} e^{-i ky} 2\pi \sum_{j=-\infty}^{\infty}
  \frac{1}{\tau}\rho(\frac{y-2\pi j}{\tau})\,\dy\,=\,
\int_{-\infty}^{\infty} e^{-i \tau k y} \rho(y)\,dy
  \,=\, \hat{\rho}(\tau k) \,.
\end{eqnarray*}
An example is $\rho(x) = \frac{1}{\sqrt{2\pi}}e^{-x^2/2}$ which gives
$\hat{\rho}(\tau k) = e^{- (\tau k)^2/2 }$. 
When $\tau$ is small, the noise is small, and when $\tau$ is large, the noise is also very
large, and $g_{\tau}$ converges to the uniform distribution when
$\tau\rightarrow\infty$. Therefore, $\gamma_1(\tau)$ is a continuous function of $\tau$ with $\gamma_1(0) = 1$ and  $\gamma_1(\tau) \to 0$ as $\tau \to \infty$.  Then $\lambda_1= \lambda_1(\tau) \leq 0$ for $\tau$ large and $\lambda_1 > 0$ for $\tau$ small. This shows that the system is linearly stable for large values of $\tau$ and unstable for small ones.

\section{An explicit example with bifurcation}
\label{sec:explicit}

The calculation here is restricted to the Maxwellian case, and we only
look for even solutions, expressed as a Fourier cosine series.
Hence we wish to solve
\begin{eqnarray}
\label{eq:coscoeffeqMaxstat}
  a_k &=&  2 \gamma_k\Gamma(k/2) a_k +  \gamma_k \sum_{n=1}^{k-1}  \Gamma(n-k/2) a_n a_{k-n} + \nonumber \\
& &2 \gamma_k \sum_{n=k+1}^{\infty}   \Gamma(n-k/2) a_n a_{n-k}
\end{eqnarray}
for $k\ge 1$. Note that $\gamma_k$ is a factor for all terms in the
right hand side, implying that if $g$ only has finitely many terms in
the Fourier series, only the corresponding terms are nonzero in~$f$.
So, here we make the following hypothesis: 

\begin{hypothesis}\label{hyp}
We assume that $g=g_{\gamma_1}$ is a family of noise distributions
with a finite number of non-zero Fourier coefficients: for some
$N<\infty$,  
$$ g_{\gamma_1}(x) = 1 + 2 \gamma_1 \, \cos x+ 2 \sum_{k=2}^N
\gamma_k(\gamma_1) \, \cos kx, \quad \forall x \in ]-\pi, \pi]. $$ 
with  $C^2$ functions $\gamma_1 \in [0,1] \mapsto \gamma_k(\gamma_1)
\in [-1,1]$ and with $\gamma_2$ such that 

$$ \gamma_2(\gamma_1) >0. $$ 
\label{hyp:lisse}
\end{hypothesis}
Note that $g$ is a probability measure as soon as $g \geq 0$. We can now state the following

\begin{theorem}
Consider a one-parameter family of noise functions $g_{\gamma_1}$
satisfying Hypothesis \ref{hyp:lisse}. Then:  
  \begin{itemize}
  \item[(i)] The uniform distribution, with Fourier coefficients $a_0=1,
    a_k=0 \;\;(k\ge1)$ is stationary. It is stable for $\gamma_1 <
    \pi/4$ and unstable for 
    $\gamma_1>\pi/4$.
  \item[(ii)] In an interval $\frac{\pi}{4}<\gamma_1 < \gamma_{max}$ there
    is another invariant solution to the dynamic problem, with Fourier
    coefficients $a_0=1, \\ a_1 = \sqrt{\frac{12 (\gamma_1-\pi/4)}{\pi
        \gamma_2(\pi/4)}}+ \bigoh((\gamma_1-\pi/4)^{3/2}) ,...., a_k=0
    \;\;(k>N)$\,. 
\item[(iii)] This solution is linearly stable with a leading eigenvalue
  $\lambda(\gamma_1)= 1 - \frac{8}{\pi}(\gamma_1-\pi/4) +
  \bigoh((\gamma_1-\pi/4)^{3/2})$. 
  \end{itemize}
\label{thm:stab_nonuniform}
\end{theorem}

\noindent
Before proving this theorem, we give a few comments. 
One is tempted to think that the same result would hold for any noise
distribution, at least provided its Fourier coefficients decay
sufficiently fast, but to prove that rigorously requires an additional
estimate showing that $\gamma_{max}$ does not converge to $\pi/4$
when the number of coefficients increases. 

We illustrate the theorem by showing numerical calculations using the family of 
noise distributions obtained as a convex combination of a Fejér kernel
and of the uniform distribution. 
\begin{align*}
  g_{\lambda}(x) =& (1-\lambda) + \lambda \frac{1}{N}\left(\frac{\sin(N x/2)
    }{x/2}\right)^2 \,.
\end{align*}
For such a noise distribution, we have  $\gamma_k= \lambda (N-k)/N$
for $1\le k <N$. Therefore, this family can be put in the framework of
Hypothesis \ref{hyp:lisse} if we link $\lambda$ to $\gamma_1$ by
$\lambda = \frac{N}{N-1} \gamma_1$. In the numerical simulations, we
use $N=9$. Fig.~\ref{fig:101} shows 
the Fourier coefficient $a_1$ as a function of the parameter
$\gamma_1$. This figure exhibits a typical pitchfork bifurcation
pattern. The order parameter $a_1$ is identically zero as long as
$\gamma_1$ is less than the critical value $\gamma_{1c} =\pi/4$ and
the associated uniform equilibrium is stable. When $\gamma_1$ becomes
larger than the critical value $\gamma_{1c}$ a second branch of
non-uniform equilibria starts. This branch is stable while the branch
of uniform equilibria becomes  unstable. In fact the non-uniform
equilibria forms a continuum, because the system is rotationally
invariant, and therefore, if $f$ is a  non-isotropic equilibrium, then any
$f(e^{i \theta_0} x)$ with $\theta_0 \in ]0,2\pi[$ is another
equilibrium. This feature is represented by the lower branch in the
diagram. In physical terms, the system exhibits a symmetry-breaking
second-order phase transition as $\gamma_1$ crosses
$\gamma_{1c}$. From the point (ii) of the theorem, it appears that the
critical exponent is $1/2$, {\em i.e.} the order parameter behaves like $a_1
\sim (\gamma_1 - \gamma_{1c})^{1/2}$ when $\gamma_1
\stackrel{\geq}{\to} \gamma_{1c}$. Fig.~\ref{fig:8} shows the noise
function $g$ and the corresponding stationary solution $f$ when
$\gamma_1= \pi/4+ 0.1$. 

\begin{figure}[h!]
  \centering
\includegraphics[width=0.7\textwidth]{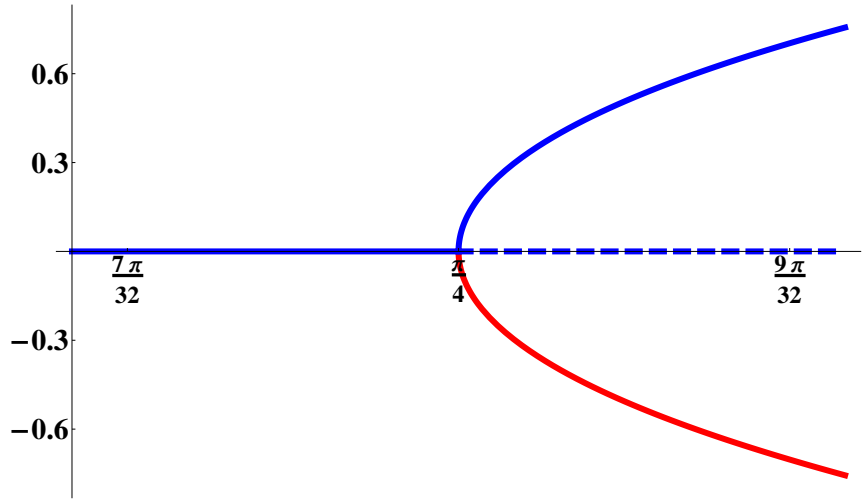}  
  \caption{The stationary solution $\bar{a}_1$ plotted as a function
    of $\gamma_1$. The noise function is a parameterized Fejér
  kernel of order 9.}  
  \label{fig:101}
\end{figure}

\begin{figure}[h!]
  \centering
\includegraphics[width=0.7\textwidth]{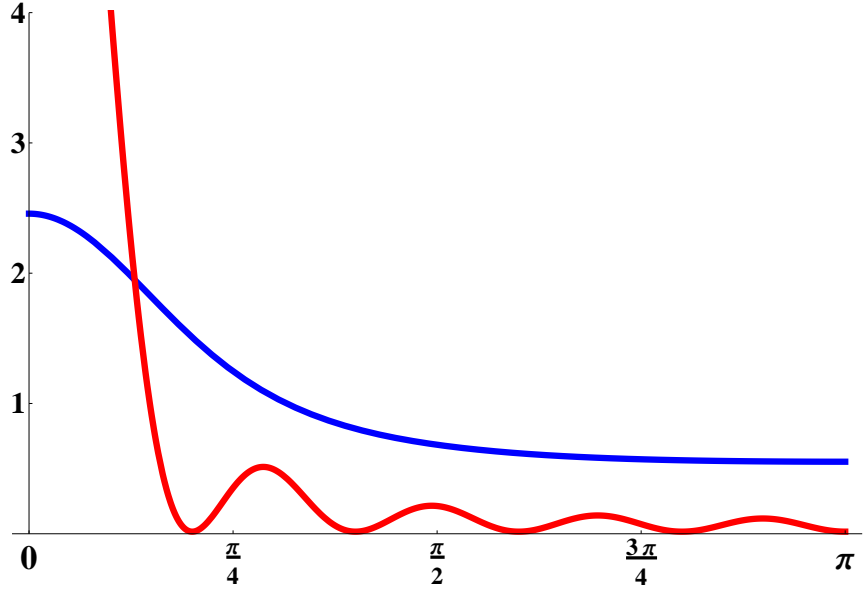}  
  \caption{The parameterized Fejér kernel of order 9 with   $\gamma_1=\pi/4+0.1$
    (red), and the corresponding solution $f(x)$ (blue)}
  \label{fig:8}
\end{figure}

\noindent
{\bf Proof of Theorem \ref{thm:stab_nonuniform}.}
The first statement, (i), is an immediate consequence of the analysis
of the linearized system in Section~\ref{sec:linearized}.

 To prove (ii) and (iii) we first note that in the Maxwellian case,
 $\Gamma(n-k/2) = 0$ when $k$ is even and different from
 $2n$. Therefore, if $k\ne0$ is even, there is only one 
non-zero term in the right hand side (\ref{eq:coscoeffeqMaxstat}) and we get: 
\begin{align*}
  a_k &= \gamma_k {a_{k/2}}^{2} \,, \quad \forall k \not = 0, \quad k \mbox{ even}. 
\end{align*}
We now concentrate on the case of $k$ odd. First, after a minor
reformulation,  
\begin{align*}
 \nonumber  a_1&=  2 \gamma_1\Gamma(1/2) a_1 + 2 \gamma_1 \Gamma(3/2) a_2 a_{1} + 
2 \gamma_1 \sum_{n=3}^{N}   \Gamma(n-1/2) a_n a_{n-1}\,,\\
\nonumber a_3 &=  2 \gamma_3\Gamma(3/2) a_3 + 2 \gamma_3\Gamma(1/2) a_2 a_1 +
   2\gamma_3 \Gamma(5/2) a_4 a_1 + 2\gamma_3 \sum_{n=5}^{N}\Gamma(n-3/2)
   a_n a_{n-3}  \\
&\vdots\\
\nonumber a_k &= 2\gamma_k\Gamma(k/2) a_k +  2\gamma_k \Gamma(1-k/2) a_{k-1} a_1
+ 2\gamma_k \sum_{n=2}^{(k-1)/2}\Gamma(n-k/2) a_n a_{k-n} + \\
\nonumber & \qquad +2\gamma_k \Gamma(1+k/2) a_{k+1} a_1 +
 2\gamma_k \sum_{n=k+2}^{N}\Gamma(n-k/2) a_n a_{n-k}
\end{align*}
Because $k$ is odd, either $n$ or $n-k$ is even. So all terms contain
a factor of the form $a_p a_q$, where $p$ is odd and $q\ge2$ is
even. Above we have separated all terms that contain a factor $a_1$.
We write $q$ in factorized form as
\begin{align*}
  q = u(q) 2^{m(q)}\equiv 2 \omega(q) \eta(q)\,
\end{align*}
with $\omega(q)$ containing all odd factors of $q$. With this notation,
\begin{align}
\label{eq:akeven}
\nonumber   a_q &= \gamma_q a_{\omega(q)2^{m(q)-1}}^2 = \gamma_q \gamma_{\omega(q)
    2^{m(q)-1}}^2 a_{\omega(q)2^{m(q)-2}}^{2^2}=...\\
&= \gamma_q
\prod_{j=1}^{m(q)-1}\gamma_{\omega(q)2^{m(q)-j}}^{2^j}a_{\omega(q)}^{2^{m(q)}}
  \equiv \tilde{\gamma}_q a_{\omega(q)}^{2\eta(q)}\,.
\end{align}
If $a_1\ne0$, we may write $a_p=a_1 \tilde{a}_p$ for all $p$ odd (this
obviously holds also for $p=1$, with $\tilde{a}_1=1$), and then 
\begin{align}
\label{eq:035}
  \frac{a_q a_p}{a_1} &=
  \tilde{\gamma_q}\gamma_2^{-\eta(q)}a_2^{\eta(q)}\tilde{a}_{\omega(q)}^{2\eta(q)}
  \tilde{a}_p   \,. 
\end{align}
Inserting these expressions in the equation for $a_1$ we get, after
dividing through by $a_1$, and using $\Gamma(x)=\frac{\sin(\pi x)}{\pi
  x}$, 
\begin{align*}
  0 &= \left(\frac{4}{\pi}\gamma_1-1\right) -\frac{4}{3\pi} \gamma_1
  a_2 +\gamma_1 R_2\equiv F_2(\gamma_1, a_2,\tilde{a}_3,\tilde{a}_5,...)\,,
\end{align*}
where $R_2$ is a sum of terms of the form~(\ref{eq:035}) with $p\ge 3$
and $q\ge 2$, {\em i.e.} monomials in $a_2$ and $\tilde{a}_p,
p=3,5,7...$ of degree at least two. Similarly the equation for $a_3$
becomes
\begin{align*}
  0 &=  \frac{4}{\pi}\gamma_3
  a_2-\left(\frac{4}{3\pi}\gamma_3+1\right)\tilde{a}_3 +\gamma_3
  R_3  \equiv F_3(\gamma_1, a_2,\tilde{a}_3,\tilde{a}_5,...)\,,
\end{align*}
where again $R_3$ is a sum of monomials of order at least two. And the
remaining equations are of the form
\begin{align*}
  0 &= \left( 2\Gamma(k/2)\gamma_k-1\right) \tilde{a}_k+\gamma_k
  R_k\equiv F_k(\gamma_1,a_2,\tilde{a}_3,\tilde{a}_5,...)\,,
\end{align*}
with $R_k$ as before. We have replaced all $\gamma_k$ by $\gamma_1$
owing to the parametrization of $\gamma_k$ by $\gamma_1$. As written
here, the functions $F_k$ depend only 
on one coefficient, $\gamma_1$. Here we also note that $\gamma_k=0$ implies that
$\tilde{a}_k=0$, and hence restricting the analysis to noise functions
with only finitely many non-zero coefficients, the system of equations
$(F_k=0)_{k=2,3,5,...}$ is reduced to a system of polynomial
equations for the unknowns $(a_2, \tilde a_3, \ldots, \tilde a_N)$,
with a right-hand side being a function of $\gamma_1$.  

We observe that at the critical value of the parameter, $\gamma_1 =
\pi/4$, the right-hand side as a function of $\gamma_1$
vanishes. Hence, the polynomial system has no degree zero term and is
solved by 
$a_2=\tilde{a}_3=\tilde{a}_3=...=\tilde{a}_n=0$. The implicit function
theorem then implies that for a sufficiently small interval around
 $\gamma_1 = \pi/4$, there is a solution $a_2(\gamma_1),
 \tilde{a}_3(\gamma_1), ...,  \tilde{a}_3(\gamma_1)$ if the Jacobian
 \begin{align*}
J&=  \left(
     \begin{array}{ccccc}
\frac{\partial F_2}{\partial a_2} & \frac{\partial F_2}{\partial \tilde{a}_3}
& ... & \frac{\partial F_2}{\partial \tilde{a}_N} \\ 
\\
\frac{\partial F_3}{\partial a_2} & \frac{\partial F_3}{\partial \tilde{a}_3}
& ... & \frac{\partial F_3}{\partial \tilde{a}_N}  \\
\vdots & \vdots &  & \vdots
\\
\frac{\partial F_N}{\partial a_2} & \frac{\partial F_N}{\partial \tilde{a}_3}
& ... & \frac{\partial F_N}{\partial \tilde{a}_N}       
     \end{array}
  \right) \\
\\
&=
 \left(
     \begin{array}{ccccc}
-\frac{4\gamma_1}{3\pi} + \gamma_1 \frac{\partial R_2}{\partial a_2}&
\gamma_1\frac{\partial R_2}{\partial \tilde{a}_3} 
& ... & \gamma_1 \frac{\partial R_2}{\partial \tilde{a}_N} \\ 
\\
\frac{4\gamma_3}{\pi}+\gamma_3\frac{\partial R_3}{\partial a_2} &
-\left(\frac{4\gamma_3}{3\pi}+1\right) + \gamma_3 \frac{\partial R_3}{\partial \tilde{a}_3}
& ... & \gamma_3 \frac{\partial R_3}{\partial \tilde{a}_N}  \\
\vdots & \vdots &  & \vdots
\\
\gamma_N\frac{\partial R_N}{\partial a_2} & \gamma_N\frac{\partial R_N}{\partial \tilde{a}_3}
& ... & 2(\gamma_N \sinc(\pi N/2)-1) +\gamma_N\frac{\partial R_N}{\partial \tilde{a}_N}       
     \end{array}
  \right)
 \end{align*}
is invertible at $\gamma_1=\frac{\pi}{4},
a_2=\tilde{a}_3=...=\tilde{a}_N=0$. Because all the $R_k$ are
polynomials of degree greater than two, we find that at the critical
point
\begin{align*}
  J&= 
 \left(\begin{array}{ccccc}
-\frac{1}{3}& 0  & ... &  0\\ 
\\
\frac{4\gamma_3}{\pi}& -\left(\frac{4\gamma_3}{3\pi}+1\right)  & ... & 0 \\
\vdots & \vdots &  & \vdots \\
0 & 0 & ... & 2(\gamma_N \sinc(\pi N/2)-1)\end{array}
  \right)
\end{align*}
Moreover, since, as seen before, the $R_k$'s are sums of monomials in
$(a_2, \tilde a_3, \ldots, \tilde a_N)$ of degree at least two, and
thanks to the assumption that $\gamma_k(\gamma_1)$ is $C^1$, we have: 
\begin{align*}
  \left(\frac{\partial F_2}{\partial \gamma_1}\right)_{\gamma_1=\frac{\pi}{4},
a_2=\tilde{a}_3=...=\tilde{a}_N=0 }&= \frac{4}{\pi}\\
\left( \frac{\partial F_k}{\partial \gamma_1}\right)_{\gamma_1=\frac{\pi}{4},
a_2=\tilde{a}_3=...=\tilde{a}_N=0 } &= 0\qquad\qquad k=3,5,...,N\,.
\end{align*}
The implicit function theorem then implies that sufficiently near
$\gamma=\pi/4$, the polynomial system can be solved, and that the
solutions $a_2,\tilde{a}_3, \tilde{a}_5,...,
\tilde{a}_N$ are differentiable functions of $\gamma_1$, with 
\begin{align*}
 \frac{d}{d\gamma_1} \left(
   \begin{array}{c}
     a_2\\ \tilde{a}_3\\ \vdots \\ \tilde{a}_N
   \end{array}\right) &= J^{-1}  \frac{d}{d \gamma_1} \left(\begin{array}{c}
     F_2\\ F_3\\ \vdots \\ F_N
   \end{array}\right)\,,
\end{align*}
where all derivatives in the right hand side are to be evaluated at
the critical point. Computing the inverse of the Jacobian, we find
easily that $a_2'(\pi/4)=12/\pi$, and with a little more effort
that $\tilde{a}_3'(\pi/4)= \frac{144 \gamma _3}{4 \pi  \gamma _3+3 \pi
    ^2}$, and then that $\tilde{a}_k'(\pi(4)=0$ for $k>3$. Hence
\begin{align*}
  a_2(\gamma_1) &= \frac{12}{\pi}\left(\gamma_1-\frac{\pi}{4}\right) +
  \bigoh\left(\left(\gamma_1-\frac{\pi}{4}\right)^2\right)\,,\\ 
  \tilde{a}_3(\gamma_1) &= \frac{144 \gamma _3}{4 \pi  \gamma _3+3 \pi
    ^2} \left( \gamma_1-\frac{\pi}{4}\right) +
  \bigoh\left(\left(\gamma_1-\frac{\pi}{4}\right)^2\right) \,,\\
 \tilde{a}_k(\gamma_1)
 &=\bigoh\left(\left(\gamma_1-\frac{\pi}{4}\right)^2\right)\,,
 \qquad\qquad k=5,7,9,...
\end{align*}


The Fourier coefficients $a_1,....,a_N$ of a stationary solution may now be
computed directly from
$a_2(\gamma_1),\tilde{a}_3(\gamma_1),...,\tilde{a}_N(\gamma_1)$ using
$a_p= a_1\tilde{a}_p$ and Eq.~(\ref{eq:akeven}).  
Because $a_2=\gamma_2 a_1^2$ and $\gamma_2>0$, and because we expect all coefficients
$a_1,...,a_N$ to be real, only $\gamma_1\ge\pi/4$ yields an admissible
solution. All coefficients are continuous functions of $\gamma_1$, and
therefore when $\gamma_1-\pi/4$ is sufficiently small, the Fourier
cosine series with these coefficients is non-negative.
Interestingly the behavior of $a_2$ near the critical point
is completely independent of the other coefficients of the noise
function than $\gamma_2$.

The uniform distribution, with $a_k=0, \; k=1,2,3....$ is always a
stationary solution, and the linearized analysis from Section~\ref{sec:linearized} showed that this solution is stable for $\gamma_1< \pi/4$ and
unstable for $\gamma_1>\pi/4$. The analysis in this section shows that
in an interval $\pi/4 < \gamma_1 < \gamma_{max}$ there is a new
invariant solution defined by the coefficients
$\bar{a}_1(\gamma_1),....,\bar{a}_N(\gamma_1)$ defined as above. 
 It now remains to prove that this new
solution is linearly stable. 
Setting $\mathbf{a}(t)=(a_1(t),
a_2(t),...,a_N(t))^{tr}$, we may write
 Eq. (\ref{eq:coscoeffeq}) as
\begin{align*}
  \frac{d}{dt}\mathbf{a}(t) =& Q(\gamma_1: \mathbf{a}(t)) - \mathbf{a}(t)\,,
\end{align*}
where $Q(\gamma_1: \mathbf{a})$ is a vector whose $k$-th element is
given by the right hand side of Eq.~(\ref{eq:coscoeffeqMaxstat}). To
prove linear stability of the 
stationary distributions $\bar{\mathbf{a}}(\gamma_1)$ computed from
above amounts 
to proving that the eigenvalues of the Jacobian matrix
$$   \frac{\partial}{\partial \mathbf{a}} Q(\gamma_1:
\bar{\mathbf{a}}(\gamma_1)) = \left( 
     \frac{\partial}{\partial a_j} Q_k(\gamma_1:
     \bar{\mathbf{a}}(\gamma_1))  \right)_{j,k=1}^N  
$$all lie inside the unit circle. The characteristic polynomial is
\begin{align*}
  p(\gamma_1,\lambda) = \det\left( \frac{\partial}{\partial
      \mathbf{a}} Q(\gamma_1: \bar{\mathbf{a}}(\gamma_1))-\lambda
    I\right)\,. 
\end{align*}
At $\gamma_1=\pi/4, a_1=...=a_N=0$,  $\frac{\partial}{\partial
  \mathbf{a}} Q(\gamma_1: \bar{\mathbf{a}}(\gamma_1))$ is a 
diagonal matrix whose diagonal entries are the coefficients $\lambda_k
+1$ as determined by (\ref{eq:stab}). They are explicitly given here
by: 
\begin{align*}
  1,\quad 0, \quad -\frac{4}{3 \pi/2}\gamma_3,\quad 0, \quad
  \frac{4}{5\pi/2} \gamma_5 ,..... 
\end{align*}
They all lie inside the unit circle except the first one. They are
continuous functions of $\gamma_1$ Therefore, as $\gamma_1$ is moved
around the critical value $\pi/4$ by a small amount, they all stay
within the unit circle, except may be the first one, which are are
going to study now. We note that, $\lambda=1$ is a simple eigenvalue
at this point: 
\begin{align*}
  p(\frac{\pi}{4}, 1) &= 0\,.
\end{align*}
We will now again use the implicit function theorem to show that
there is a function $\lambda(\gamma_1)$ such that $\lambda(\pi/4)=1$,
$ p(\gamma_1,\lambda(\gamma_1)) = 0$,  and  
\begin{align}
\label{eq:detderivative1}
  \lambda'(\frac{\pi}{4})&= - \left(\frac{\partial
      p(\gamma_1,\lambda)}{\partial
      \lambda}\right)^{-1}_{\gamma_1=\frac{\pi}{4}, \lambda=1}\left(\frac{\partial
      p(\gamma_1,\lambda)}{\partial \gamma_1}
  \right)_{\gamma_1=\frac{\pi}{4}, \lambda=1}  = -\frac{8}{\pi}
\end{align}
This implies that for
$\gamma_1>\pi/4$, sufficiently small, $|\lambda(\gamma_1)| < 1$,
and that $\bar{\mathbf{a}}(\gamma_1)$ is a stable (hyperbolic) fixed
point for the system in Eq.~(\ref{eq:coscoeffeq}) in the Maxwellian
case and with $N$ non-zero noise coefficients $\gamma_k$.

To obtain (\ref{eq:detderivative1}) we write
$\frac{\partial}{\partial \mathbf{a}} Q(\gamma_1: \bar{\mathbf{a}}(\gamma_1))-\lambda I$ in more detail. Explicitly
for 5 non-zero coefficients $\gamma_k$, this matrix is equal to:
\begin{align*}
\left(
\begin{array}{ccccc}
 -\frac{\displaystyle 4 a_2 \gamma_1}{\displaystyle 3 \pi }+
\frac{\displaystyle 4 \gamma_1}{\displaystyle \pi }-\lambda  & 2 
 \left(\frac{\displaystyle 2 a_3}{\displaystyle 5 \pi}-
   \frac{\displaystyle 2 a_1}{\displaystyle 3 \pi }\right) \gamma_1
 & 2 \left(\frac{\displaystyle 2 a_2}{\displaystyle 5 \pi }-
    \frac{\displaystyle 2 a_4}{\displaystyle 7 \pi }\right) 
   \gamma_1 & 2 \left(\frac{\displaystyle 2 a_5}{\displaystyle 9
       \pi}-
    \frac{\displaystyle 2 a_3}{\displaystyle 7 \pi }\right) \gamma_1 & \frac{\displaystyle 4 a_4 \gamma_1}{\displaystyle 9 \pi } \\  \\
 2 a_1 \gamma_2 & -\lambda  & 0 & 0 & 0 \\ \\
 \frac{\displaystyle 4 a_2 \gamma_3}{\displaystyle \pi }+
  \frac{\displaystyle 4 a_4 \gamma_3}{\displaystyle 5 \pi } & 
  \frac{\displaystyle 4 a_1 \gamma_3}{\displaystyle \pi }-
   \frac{\displaystyle 4 a_5 \gamma_3}{\displaystyle 7 \pi } & 
    -\frac{\displaystyle 4 \gamma_3}{\displaystyle 3 \pi }-\lambda  &
     \frac{\displaystyle 4 a_1 \gamma_3}{\displaystyle 5 \pi } & 
   -\frac{\displaystyle 4 a_2 \gamma_3}{\displaystyle 7 \pi } \\ \\
 0 & 2 a_2 \gamma_4 & 0 & -\lambda  & 0 \\ \\
 -\frac{\displaystyle 4 a_4 \gamma_5}{\displaystyle 3 \pi } & 
   \frac{\displaystyle 4 a_3 \gamma_5}{\displaystyle \pi } &
 \frac{\displaystyle 4 a_2 \gamma_5}{\displaystyle \pi } & 
   -\frac{\displaystyle 4 a_1 \gamma_5}{\displaystyle 3 \pi } & 
   \frac{\displaystyle 4 \gamma_5}{\displaystyle 5 \pi }-\lambda
\end{array}
\right)
\end{align*}
Substituting $\gamma_1$ with $\pi/4+\tau$ and $\lambda$ with
$1+\mu$ we find, retaining only the lowest order terms in each
coefficient and only coefficients of order one or less in $\tau$
and $\mu$, 
\begin{align*}
  \left(
\begin{array}{ccccc}
 -\mu & -\frac{\displaystyle 4 a_1 \tau}{\displaystyle 3
   \pi }-\frac{\displaystyle a_1}{\displaystyle 3} & 
 \frac{\displaystyle 12 \tau}{\displaystyle 5 \pi } & 0 & 0 \\ \\
 2 a_1 \bar{\gamma}_2 & -1 & 0 & 0 & 0 \\ \\
 \frac{\displaystyle 48 \tau \bar{\gamma}_3}{\displaystyle \pi ^2} &
 \frac{\displaystyle 4 a_1 \bar{\gamma}_3}{\displaystyle \pi } 
 & -\frac{\displaystyle 4 \bar{\gamma}_3}{\displaystyle 3 \pi }-1 
   & \frac{\displaystyle 4 a_1 \bar{\gamma}_3}{\displaystyle 5 \pi } &
   -\frac{\displaystyle 48 \tau \bar{\gamma}_3}{\displaystyle 7 \pi
     ^2} \\ \\
 0 & \frac{\displaystyle 24 \tau \bar{\gamma}_4}{\displaystyle \pi } &
 0 & -1 & 0 \\ \\
 0 & 0 & \frac{\displaystyle 48 \tau \bar{\gamma}_5}{\displaystyle \pi
   ^2} & -\frac{\displaystyle 4 a_1\bar{\gamma}_5}{\displaystyle 3 \pi } & \frac{\displaystyle 4
   \bar{\gamma}_5}{\displaystyle 5 \pi }-1 \,.  
\end{array}
\right)
\end{align*}
In this expression $\bar{\gamma}_k = \gamma_k(\pi/4)$. It is easy to see that this matrix has essentially the same form for
any number of non-zero coefficients $\gamma_k$, a five-diagonal matrix
where the diagonal elements except the first one are of order $\bigoh(1)$ and
all other elements are $\bigoh( \mu+\tau^{1/2})$ (because $\bar a_1 \sim \bar a_2^{1/2} = \bigoh(\tau^{1/2})$). Hence,
expanding the determinant, we find, after some computation, that
\begin{align*}
  p(\frac{\pi}{4}+\tau, 1+\mu) &= C_N
  (\mu+\frac{2}{3}\bar{\gamma}_2 a_1^2) + \bigoh(
  \mu^2+\tau^{3/2})\\
&= C_N
  (\mu+\frac{2}{3}a_2) + \bigoh(
  \mu^2+\tau^{3/2})\\
&= C_N
  (\mu+\frac{8}{\pi}\tau) + \bigoh(
  \mu^2+\tau^{3/2})\,
\end{align*}
where $C_N$ is the product of the diagonal elements from row three and below.
And we conclude, as stated in eq.~(\ref{eq:detderivative1}) that 
\begin{align*}
 -  \frac{\partial p}{\partial \gamma_1} /  \frac{\partial p}{\partial
    \lambda} = - 8 / \pi\,,
\end{align*}
when evaluated at the critical point $\gamma_1=\pi/4, \lambda=1$. Again
we note that this is independent of the Fourier coefficients of the
noise function.

\section{The method of partitions of integers by Ben-Naim and Krapivsky}
\label{sec:BenNaimKrapivsky}

In this section we adapt a method of  Ben-Naim and 
Krapivsky~\cite{BenNaimKrapivsky2006} to the construction of invariant densities
for our equation in the Maxwellian case.  We no longer require Hypothesis~\ref{hyp}, but on the other hand, we 
shall not control the convergence of infinite sums, and our conclusions are therefore formal. Nonetheless,
as in  \cite{BenNaimKrapivsky2006}, the method provides another view of the phase transition studied here.

With $\gamma_k$ defined as above and  $\Gamma(u)=\sin(\pi
u)/(\pi u)$, we let
\begin{eqnarray*}
  G_{i,j} &=& \frac{\gamma_{i+j}}{1-2 \gamma_{i+j}\Gamma\left(\frac{i+j}{2}\right)}\Gamma\left(\frac{i-j}{2}\right)\,,
\end{eqnarray*}
which is defined for $i,j\in\Z$. Clearly
\begin{eqnarray*}
  G_{i,j}=G_{j,i},\quad  G_{i,j}=G_{-i,-j},\qquad
\mbox{and}\qquad G_{j,j}=\gamma_{2j}\,.
\end{eqnarray*}
Also
\begin{equation}
  \label{eq:c1.6}
  G_{i,j} = 0\quad\mbox{when} \quad(i-j)\ne0\quad\mbox{is even}\,,
\end{equation}
whereas for $j-i$ odd, $G_{i,j}$ satisfies
$$  |G_{i,j}| \le
  \frac{|\gamma_{i+j}|}{1-4|\gamma_{i+j}|/(\pi|i+j|)}\frac{2}{\pi|i-j|} \,. $$

Because we only look for even solutions,
$a_j=a_{-j}$, equation~(\ref{eq:coscoeffeq}) may now  be written
\begin{eqnarray}
\label{eq:eq34}
  a_k&=& \sum_{j=1}^{k-1} G_{k-j,j} a_{k-j}a_j + 2\sum_{j=1}^{\infty}
  G_{k+j,-j} a_{k+j} a_j\,.
\end{eqnarray}
It follows from~(\ref{eq:c1.6}) and (\ref{eq:eq34}) that when $k$ is a power of two, one can express $a_k$ in terms
of $a_1=a_{-1}$. Hence
with $k=2^m$, 
\begin{eqnarray*}
  a_{2^m}&=& \gamma_{2^m} \left(a_{2^{m-1}}\right)^2\,,
\end{eqnarray*}
and iterating gives
\begin{equation}
\label{eq:c1.9}
  a_{2^m}= \prod_{j=0}^{m-1} \left(\gamma_{2^{m-j}}\right)^{2^j} a_{1}^{2^m}\,.
\end{equation}

One might hope that it is possible to express {\em every} $a_k$ as, if not a polynomial in $a_1$, at least
as a power series in $a_1$. 
The strategy in ~\cite{BenNaimKrapivsky2006} provides such an expression, and $a_1$
itself is considered an
{\em order parameter} and denoted $R$: for $k\ge 2$,
\begin{eqnarray}
\label{eq:c1.11}
  a_k &=& \sum_{n=0}^{\infty} p_{k,n} R^{|k|+2n}\,,
\end{eqnarray}
where the the coefficients $p_{k,n}$ are a sum of various products of
$G_{i,j}$ computed using a generalized integer partition of $k$ as a
sum of $k+n$ terms of $+1$ and $n$ terms of $-1$. The formula
corresponding to~(\ref{eq:c1.11}) in~\cite{BenNaimKrapivsky2006} is
written with $k$ instead of $|k|$ in the exponent of $R$, and this
leads to the erroneous formula (15) in their paper. We will now
derive a correct replacement of their formula (15) adapted to our
case.

\subsection{The recursion formula}

Here we look for an invariant density $f$ whose Fourier
coefficients, $a_k$ ($k\ge2$) are given by a power series in $R$ of
the form~(\ref{eq:c1.11}), using, of course, $k=|k|$. For $a_1$, there
is such a representation,
\begin{equation}
\label{eq:c1.12}
  p_{1,n} = \delta_{n,0} = \left\{
      \begin{array}{ll}
        1\quad\mbox{if  }&n=0\\
        0      \quad\mbox{if  }          &n\ne0
      \end{array}
 \right.
\end{equation}
but we will also use a different representation in which
$p_{1,0}=0$. Combining the two expressions gives the equation
\begin{equation*}
  R=\sum_{n=0}^{\infty} p_{1,n} R^{1+2n}\,,
\end{equation*}
from which the value of $R$ can be determined. Clearly, $R=0$ is a
solution, corresponding to the uniform distribution $f=(2\pi)^{-1}$.

\begin{lemma}
  \label{lem:p_relation}
  For each positive integer $k$, let $\{p_{k,n}\}$ be a sequence of
  numbers such that the power series
  $\sum_{n=0}^{\infty}p_{k,n}z^{k+2n}$ has radius of convergence at
    least one. For $-1<R<1$, define 
    \begin{equation*}
      a_{-k}(R)=a_k(R) = \sum_{n=0}^{\infty} p_{k,n} R^{k+2n}\,.
    \end{equation*}
  Then the $a_k(R)$ satisfy~(\ref{eq:eq34}) for all $R$ and all $k\ge1$ if and only if the
  numbers $\{p_{k,n}\}$ for $k\ge 1$ and $n\ge0$ satisfy
  \begin{equation}
   \label{eq:c1.13}
    p_{k,n}= \sum_{j=1}^{k-1}\sum_{\ell=0}^{n}
    G_{k-j,j}p_{k-j,\ell}p_{j,n-\ell} + 2 \sum_{j=1}^{n}
    \sum_{\ell=0}^{n-j} G_{k+j,-j} p_{k+j,\ell} p_{j,n-(j+\ell)}\,.
  \end{equation}
Note that for $n=0$  the second sum is zero.
\end{lemma}

\begin{proof} Take $k\ge0$. 
  Substituting~(\ref{eq:c1.11}) into equation~(\ref{eq:eq34}) gives
  \begin{align}
    \label{eq:c1.15}
   \sum_{n=1}^{\infty} p_{k,n} R^{k+2n} =&
   \sum_{j=1}^{k-1}\sum_{\ell=0}^{\infty}\sum_{m=0}^{\infty}
   R^{k+2(\ell+m)} G_{k-j,j} \, p_{k-j,\ell}\, p_{j,m} \nonumber \\
    &+ 2  \sum_{j=1}^{\infty}\sum_{\ell=0}^{\infty}\sum_{m=0}^{\infty}
        R^{k+2(j+\ell+m)} G_{k+j,-j} \, p_{k+j,\ell}\, p_{j,m}  \,.
  \end{align}
  Equating coefficients of like powers of $R$, we
  obtain~(\ref{eq:c1.13}). Conversely, if~(\ref{eq:c1.13}) is
  satisfied for all $k\ge2$, then~(\ref{eq:c1.15}) is also satisfied
  for $k\ge2$. 
\end{proof}

As the proof of the lemma show, if we could find numbers $p_{k,n}$  such that 
(\ref{eq:c1.13}) is satisfied for all $k\geq 1$, then we would construct a family, parameterized
by $R$, of solutions (not necessarily positive) of the invariant measure equation. 

This, of course, is more than we expect to find, and so the lemma must be supplemented by
two things: (1) A construction  of the numbers $p_{k,n}$. (2) A mechanism for selecting a 
particular value of $R$.

Following ~\cite{BenNaimKrapivsky2006}, we present a recursive construction of the numbers
$p_{k,n}$, and a consistent argument for determining $R$.

\subsection{The recursion formula }

We need some known values of the $p_{k,n}$ to start the recursive construction. 
First, notice that when $k$ is a power of two, there is
only one non-zero term in the right-hand side of~(\ref{eq:c1.13}), and
a simple recursion gives 
\begin{equation}
\label{eq:c1.17}
  p_{2^m,n} = \prod_{j=0}^{m-1}\left(\gamma_{2^{m-j}}\right)^{2^j} \,\delta_{n,0}\,,
\end{equation}
which is consistent with~(\ref{eq:c1.9}). 

On the other hand,
equation~(\ref{eq:c1.13}) is inconsistent
with~(\ref{eq:c1.12}). Indeed, for $k=1$, the first 
sum in~(\ref{eq:c1.13}) is zero because the range of summation is
empty. Then for $n=0$ also the second sum is zero, so,
$p_{1,0}=0$. This can be seen already in~(\ref{eq:c1.15}), because
there, in the right hand side, the smallest power of $R$ that is
present is $R^{1+2(j-m-\ell)}$ with $j=1$ and  $m=\ell=0$, {\em i.e.}
$R^3$. However, the coefficient of $R^3$ is a multiple of $p_{1,0}$,
so $p_{1,1}=0$ as well. Hence the first non-vanishing coefficient for
$a_1$ is $p_{1,2}$.

This discrepancy is the source of the criterion for selecting a particular value of $R$
that yields an invariant density. 

To start the recursive determination of the coefficients, note that 
when $n=0$, 
the range in the second sum in (\ref{eq:c1.13}) is empty. Thus,
we have
 $$   p_{k,0}= \sum_{j=1}^{k-1}
    G_{k-j,j}p_{k-j,0}p_{j,0}\ . $$
 Since as noted above $p_{1,0} =1$ and $p_{2,0} = \gamma_2$, $p_{3,0}$
 is determined and then, recursively,  so is $p_{k,0}$ for all $k$.

Next, we consider  $p_{k,n}$ for  $k=1$. 
Specializing~(\ref{eq:c1.13}) to $k=1$, we obtain
$$  p_{1,n}\,=\,2\sum_{j=1}^{n}\sum_{\ell=0}^{n-j} G_{1+j,-j}\,p_{1+j,\ell}\,p_{n-(j+\ell)}\,. $$
The first two terms in this sequence are
\begin{equation*}
  p_{1,2}\,=\,2G_{3,-2}\,p_{3,0}\,p_{2,0}\,,
\end{equation*}
and
\begin{equation*}
  p_{1,3}\,=\,2\left(G_{2,-1} p_{2,0}\,p_{1,2} + G_{3,-2}
    p_{3,1}\,p_{2,0} \right)\,.
\end{equation*}
Here we have used $p_{1,0}=p_{1,1}=p_{2,1}=0$, the latter being true
because of~(\ref{eq:c1.17}), which reduces to $p_{2,n}=\gamma_2
\delta_{n,0}$ when $k=2$. 
All terms in the expression for $p_{1,2}$ have been determined above.
To compute $p_{1,3}$, we need $p_{3,1}$. However,
 $$   p_{3,1}= \sum_{j=1}^{2}\sum_{\ell=0}^{1}
    G_{k-j,j}p_{k-j,\ell}p_{j,1-\ell} + 2 
     G_{4,-1} p_{4,0} p_{1,0}\ .$$
Since $p_{4,0}$ is known, we have $p_{3,1}$ and hence $p_{1,3}$.   So far, we have determined the values of
all $p_{k,n}$ for all $k+n \leq 4$, and then some. 
From here it is not hard to see that  the values of all of the $p_{k,n}$ are determined.
For a discussion of this in terms of integer partitions, see ~\cite{BenNaimKrapivsky2006}.
Though all of the coefficients are determined, it does not seem to be a simple matter to
estimate the size of the coefficients in a manner that is useful for proving that they do define power series with even a positive radius of convergence.

\subsection{The consistency condition}

At this stage, we have the coefficients $p_{k,n}$ for all $k\geq 1$
and all $n\geq0$. The equations 
(\ref{eq:c1.13}) are satisfied for all $k\geq 1$, by construction, but
not, as we have pointed out, for 
$k=1$ by the coefficients given in (\ref{eq:c1.12}), which corresponds
to $a_1(R) = R$ for all $-1 < R < 1$.   

Nonetheless, assuming convergence, we have from (\ref{eq:c1.9}) that $R =a_1$. 
Using the coefficients derived above, we have
$$a_1(R) = \sum_{0}^\infty p_{1,n}R^{1+2n}\ ,$$
and the first non-vanishing term in the power series on the right is
for $n=2$, so that $a_1(R) \sim R^5$ at $R=0$. 

Therefore, any value of $R$ giving an invariant measure must satisfy
$$R = a_1(R) \ ,$$
where $a_1(R)$ is the function defined by the power series derived above. 
Of course, there is always the solution $R=0$. However, there may be
other solutions.  
In~\cite{BenNaimKrapivsky2006}, the function $a_1(R)$ is approximately
computed numerically  
and plotted. For noise parameters such that  $R = a_1(R)$ has a
non-zero solution, they 
find a non-trivial invariant measure. However, rigorous analysis of
this construction, and especially 
analysis of stability of the invariant measures so constructed, seems
difficult, and this has motivated our 
different treatment. While less general in its scope, due to
Hypothesis~\ref{hyp}, it does permit rigorous analysis.

\section{Conclusion}
\label{sec:conclu}

In this paper, we have studied a Boltzmann model intended to provide a binary interaction description of alignment dynamics which appears in swarming models such as the Vicsek model. In this model, pairs of particles lying on the circle interact by trying to reach their mid-point up to some noise. We have studied the equilibria of this Boltzmann model and, in the case where the noise probability has only a finite number of non-zero Fourier coefficients, rigorously shown the existence of a pitchfork bifurcation as a function of the noise intensity. In the case of an infinite number of non-zero Fourier modes, we have adapted a method proposed by Ben-Na\"im and Krapivsky to show (at least formally) that a similar behavior can be obtained. In the future, we expect to be able to show the rigorous convergence of the infinite series involved in the Ben-Na\"im and Krapivsky argument, and therefore, to give a solid mathematical ground also to this case. Extensions of the model to higher dimensional spheres or other manifolds is also envisionned. Finally, the non-isotropic equilibria found beyond the critical threshold will allow us to develop non-trivial Self-Organized Hydrodynamics, as done earlier in the case of the Vicsek mean-field dynamics.

\bibliographystyle{plain}


\begin{thebibliography}{1}

\bibitem{Aldana_etal_PRL07}
M. Aldana, V. Dossetti, C. Huepe, V. M. Kenkre and H. Larralde. 
\newblock Phase transitions in systems of self-propelled agents and related network models.
\newblock {\em Phys. Rev. Lett.}, 98:095702, (2007).

\bibitem{Aoki_BullJapSocSciFish92} I. Aoki. 
\newblock A simulation study on the schooling mechanism in fish. 
\newblock {\em Bulletin of the Japan Society of Scientific Fisheries}, 48:1081-1088, (1982).

\bibitem{Barbaro_Degond_DCDSB13}
A. Barbaro and P. Degond.
\newblock  Phase transition and diffusion among socially interacting self-propelled agents.
\newblock {\em  Discrete Contin. Dyn. Syst. Ser. B}, to appear. 

\bibitem{Baskaran_Marchetti_PRE08} 
A. Baskaran and M. C. Marchetti.
\newblock Hydrodynamics of self-propelled hard rods. 
\newblock {\em Phys. Rev. E}, 77:011920 (2008).

\bibitem{Baskaran_Marchetti_PRL10} 
A. Baskaran and M. C. Marchetti.
\newblock Nonequilibrium statistical mechanics of self-propelled hard rods. 
\newblock {\em J. Stat. Mech. Theory Exp.}, P04019, (2010). 

\bibitem{Bellomo_Soler_M3AS12}
N. Bellomo and J. Soler.
\newblock  On the mathematical theory of the dynamics of swarms viewed as complex systems.
\newblock {\em Math. Models Methods Appl. Sci.}, 22, Supp1:1140006, (2012).

\bibitem{BenNaimKrapivsky2006}
E.~Ben-Naim and P.~L. Krapivsky.
\newblock Alignment of rods and partition of integers.
\newblock {\em Phys. Rev. E}, 73(3):031109, (2006).

\bibitem{Bertin_etal_NewJPhys13}
E. Bertin, H. Chat\'e, F. Ginelli, S. Mishra, A. Peshkov and S. Ramaswamy.
\newblock  Mesoscopic theory for fluctuating active nematics.
\newblock {\em New J. Phys.}, 15:085032, (2013).

\bibitem{BertinDrozGregoire2006}
E.~Bertin, M.~Droz and G.~Gr{\'e}goire.
\newblock Boltzmann and hydrodynamic description for self-propelled particles.
\newblock {\em Phys. Rev. E}, 74:022101, (2006).

\bibitem{Bertin_etal_JPhysA09}
E.~Bertin, M.~Droz and G.~Gr{\'e}goire.
\newblock Hydrodynamic equations for self-propelled particles: microscopic derivation and stability analysis.
\newblock {\em J. Phys. A: Math. Theor.} 42:445001, (2009) . 

\bibitem{Bolley_etal_M3AS11} 
F. Bolley, J. A. Ca\~ nizo and J. A. Carrillo. 
\newblock Stochastic Mean-Field Limit: Non-Lipschitz Forces \& Swarming. 
\newblock {\em Math. Models Methods Appl. Sci.}, 21:2179-2210, (2011).

\bibitem{CarlenChatelinDegondWennberg2011}
E. Carlen, R. Chatelin, P. Degond and B. Wennberg.
\newblock Kinetic hierarchy and propagation of chaos in biological swarm models.
\newblock {\em Physica D, Nonlinear phenomena},  260:90-111, (2013).

\bibitem{CarlenDegondWennberg2011}
E. Carlen, P. Degond and B. Wennberg.
\newblock Kinetic limits for pair-interaction driven master equations and biological swarm models.
\newblock 	{\em Math. Models and Methods in Appl  Sci.}, 23(7):1339-1376, (2013).

\bibitem{Carrillo_etal_SIMA10} J. A. Carrillo, M. Fornasier, J. Rosado and G. Toscani. 
\newblock Asymptotic Flocking Dynamics for the kinetic Cucker-Smale model.
\newblock {\em SIAM J. Math. Anal.} 42:218-236, (2010).

\bibitem{Chate_etal_PRE08}
H. Chat\'e, F. Ginelli, G. Gr\'egoire and F. Raynaud. 
\newblock Collective motion of self-propelled particles interacting without cohesion.
\newblock {\em Phys. Rev. E}, 77:046113, (2008).

\bibitem{Chuang_etal_PhysicaD07} 
Y-L. Chuang, M. R. D'Orsogna, D. Marthaler,  A. L. Bertozzi and L. S. Chayes. 
\newblock State transitions and the continuum limit for a 2D interacting, self-propelled particle system. 
\newblock {\em Physica D}, 232:33-47, (2007). 


\bibitem{Cordier_etal_JSP05} S. Cordier, L. Pareschi and G. Toscani. 
\newblock On a kinetic model for a simple market economy. 
\newblock 	{\em J. Stat. Phys.}, 120:253-277, (2005) . 

\bibitem{Couzin_etal_JTB02} 
I. D. Couzin, J. Krause, R. James, G. D. Ruxton and N. R. Franks. 
\newblock Collective Memory and Spatial Sorting in Animal Groups.
\newblock {\em  J. theor. Biol.}, 218:1-11, (2002 . 

\bibitem{Cucker_Smale_IEEETransAutCont07}
F. Cucker and S. Smale. 
\newblock Emergent behavior in flocks.
\newblock {\em  IEEE Transactions on Automatic Control}, 52:852-862, (2007).

\bibitem{Czirok_etal_PRE96}
A. Czir\`ok, E. Ben-Jacob, I. Cohen and T. Vicsek. 
\newblock Formation of complex bacterial colonies via self-generated vortices.
\newblock {\em  Phys. Rev. E}, 54:1791-1801, (1996) .

\bibitem{Degond_etal_JNonlinearSci13}
P. Degond, A. Frouvelle and J-G. Liu. 
\newblock Macroscopic limits and phase transition in a system of self-propelled particles. 
\newblock {\em J. Nonlinear Sci.}, 23:427-456, (2013). 

\bibitem{Degond_etal_preprint13}
P. Degond, A. Frouvelle and J-G. Liu. 
\newblock Phase transitions, hysteresis, and hyperbolicity for self-organized alignment dynamics.
\newblock {\em submitted.}  arXiv:1304.2929.

\bibitem{Degond_etal_Schwartz13}
P. Degond, A. Frouvelle, J.-G. Liu, S. Motsch and L. Navoret. 
\newblock Macroscopic models of collective motion and self-organization.
\newblock {\em S\'eminaire Laurent Schwartz - EDP et applications} , 1, (2012-2013).

\bibitem{Degond_etal_arXiv:1403.5233}
P. Degond, A. Frouvelle, G. Raoul.
\newblock Local stability of perfect alignment for a spatially homogeneous kinetic model. 
\newblock {\em Submitted.} arXiv:1403.5233. 

\bibitem{Degond_etal_MAA13}
P. Degond, J-G. Liu, S. Motsch and V. Panferov. 
\newblock Hydrodynamic models of self-organized dynamics: derivation and existence theory.
\newblock {\em  Methods Appl. Anal.}, 20:089-114, (2013).

\bibitem{Degond_Motsch_M3AS08}
P. Degond and S. Motsch. 
\newblock Continuum limit of self-driven particles with orientation interaction.
\newblock {\em Math. Models Methods Appl. Sci.}, 18Suppl:1193-1215, (2008). 

\bibitem{Fornasier_etal_PhysicaD11}
M. Fornasier, J. Haskovec and G. Toscani. 
\newblock Fluid dynamic description of flocking via the Povzner-Boltzmann equation.
\newblock {\em Phys. D},  240:21-31, (2011). 

\bibitem{Frouvelle_M3AS12} 
A. Frouvelle. 
\newblock A continuum model for alignment of self-propelled particles with anisotropy and density-dependent parameters.
\newblock {\em  Math. Mod. Meth. Appl. Sci.}, 22:1250011, (2012).

\bibitem{Frouvelle_Liu_SIMA12}
A. Frouvelle and J.-G. Liu. 
\newblock Dynamics in a kinetic model of oriented particles with phase transition.
\newblock {\em SIAM J. Math. Anal.}, 44:791-826, (2012). 

\bibitem{Gautrais_etal_PlosCB12}
J. Gautrais, F. Ginelli, R. Fournier, S. Blanco, M. Soria, H. Chat\'e and G. Theraulaz. 
\newblock Deciphering interactions in moving animal groups.
\newblock {\em  Plos Comput. Biol.}, 8:e1002678, (2012).

\bibitem{Gretoire_Chate_PRL04} G. Gr\'egoire and  H. Chat\'e. 
\newblock Onset of collective and cohesive motion.
\newblock {\em  Phys. Rev. Lett.}, 92:025702, (2004).

\bibitem{Ha_Liu_CMS09}
S. -Y. Ha and J.-G. Liu. 
\newblock A simple proof of the Cucker-Smale flocking dynamics and mean-field limit.
\newblock {\em  Commun. Math. Sci.}, 7:297-325, (2009).

\bibitem{Ha_Tadmor_KRM08}
S.-Y. Ha and E. Tadmor. 
\newblock From particle to kinetic and hydrodynamic descriptions of flocking.
\newblock {\em  Kinet. Relat. Models}, 1:415-435, (2008).

\bibitem{Mogilner_etal_JMB03}
A. Mogilner, L. Edelstein-Keshet, L. Bent and A. Spiros. 
\newblock Mutual interactions, potentials, and individual distance in a social aggregation.
\newblock {\em J. Math. Biol.}, 47:353-389, (2003). 

\bibitem{Motsch_Tadmor_JSP11}
S. Motsch and E. Tadmor. 
\newblock A new model for self-organized dynamics and its flocking behavior.
\newblock {\em  J. Stat. Phys.}, 144:923-947, (2011).

\bibitem{Peruani_etal_PRE06}
F. Peruani, A. Deutsch and M. B\"ar. 
\newblock Nonequilibrium clustering of self-propelled rods.
\newblock {\em  Phys. Rev. E},  74:030904(R), (2006).

\bibitem{Ratushnaya_etal_PhysicaA07}
V. I. Ratushnaya, D. Bedeaux, V. L. Kulinskii and A. V. Zvelindovsky. 
\newblock Collective behavior of self propelling particles with kinematic constraints: the relations between the discrete and the continuous description.
\newblock {\em Phys. A}, 381:39-46, (2007).

\bibitem{Toner_Tu_PRL95} J. Toner and Y. Tu. 
\newblock Flocks, Long-range order in a two-dimensional dynamical XY model: how birds fly together.
\newblock {\em  Phys. Rev. Lett.}, 75:4326-4329 (1995).

\bibitem{Toner_etal_AnnPhys05} J. Toner, Y. Tu and S. Ramaswamy. 
\newblock Hydrodynamics and phases of flocks.
\newblock {\em Annals of Physics}, 318:170-244, (2005). 


\bibitem{Vicsek_etal1995}
T.~Vicsek, A.~Czirok, E.~Ben Jacob, I.~Cohen and O.~Shochet.
\newblock {Novel type of phase-transition in a system of self-driven particles}.
\newblock {\em Phys. Rev.  Lett.}, 75(6):1226-1229, (1995).

\bibitem{Vicsek_Zafeiris_PhysRep12} 
T. Vicsek and A. Zafeiris. 
\newblock Collective motion.
\newblock {\em Phys. Rep.}, 517:71-140, (2012).

\end{thebibliography}

\def\cprime{$'$}

\end{document}